%% file: vol_lat_simplices.tex
\documentclass[11pt,a4paper]{article}

\title{%
	\textsc{%
		Lattice simplices with a fixed positive number of interior lattice points: \\ A nearly optimal volume bound
	}
}

\author{%
	Gennadiy Averkov\footnote{Faculty of Mathematics, Otto-von-Guericke-Universit\"at Magdeburg, Universit\"atsplatz 2, 39106 Magdeburg, Germany. Emails: averkov@ovgu.de, jan.kruempelmann@ovgu.de, benjamin.nill@ovgu.de}%
	\qquad Jan Kr\"umpelmann${}^\ast$%
	\qquad Benjamin Nill${}^\ast$
}

\usepackage{amsthm,amsmath,amssymb,enumerate,graphicx,graphpap}
\usepackage{mathtools}
\usepackage{empheq}
\usepackage{paralist}
\usepackage{color}
\usepackage{soul}
\usepackage{tikz}

\usepackage[format=plain,labelfont={bf,it},textfont=it]{caption}


\input{commands}
\input{environments}

\numberwithin{equation}{section}

\begin{document}

\maketitle

\begin{abstract}
We give an explicit upper bound on the volume of lattice simplices with fixed positive number of interior lattice points. The bound differs from the conjectural sharp upper bound only by a linear factor in the dimension. This improves significantly upon the previously best results by Pikhurko from 2001.
\end{abstract}

\input{introduction}

\input{background}

\input{prod-sum}

\input{opt-prod-sum}

\subsection*{Acknowledgements} 

The second author was supported by a scholarship of the state of Sachsen-Anhalt, Germany. The third author is an affiliated researcher of Stockholm University and partly supported by the Vetenskapsr{\aa}det grant NT:2014-3991. We thank Christian Haase, Martina Juhnke-Kubitzke and Noleen K\"ohler for their interest. We are also grateful to Gabriele Balletti for checking some conjectures in the database \cite{arXiv:1612.08918}. Both authors are PI's in the Research Training Group Mathematical Complexity Reduction funded by the German Research Foundation (DFG-GRK 2297).

%\begin{acknowledgments*} The second author was supported by a scholarship of the state of Sachsen-Anhalt, Germany. The third author was supported by the Vetenskapsr{\aa}det grant NT:2014-3991. We thank Christian Haase and Noleen K\"ohler for their interest and Alexander Kasprzyk and Gabriele Balletti for sharing the results of their computer-assisted enumeration. 
%\end{acknowledgments*}

\bibliography{lit2}
\bibliographystyle{amsalpha}

\end{document}

%% file: commands.tex
\newcommand{\ORD}{\operatorname{ORD}}
\newcommand{\PS}{\operatorname{PS}}
\newcommand{\rmcmd}[1]{\mathop{\mathrm{#1}}\nolimits}

\newcommand{\R}{\mathbb{R}}
\newcommand{\N}{\mathbb{N}}
\newcommand{\Z}{\mathbb{Z}}

\newcommand{\cP}{\mathcal{P}}

\newcommand{\cS}{\mathcal{S}}

\newcommand{\setcond}[2]{\left\{#1 \, : \, #2 \right\}}

\newcommand{\vol}{\rmcmd{vol}}

\newcommand{\vertset}[1]{\operatorname{vert}(#1)}
\newcommand{\conv}{\operatorname{conv}}
\newcommand{\intr}[1]{\operatorname{int}(#1)}

\newcommand{\Pd}[1]{\mathcal{P}^d({#1})}
\newcommand{\Sd}[1]{\mathcal{S}^d({#1})}

\newcommand{\eps}{\varepsilon}

%% file: environments.tex
\newtheorem{nn}{}[section]
\newtheorem{theorem}[nn]{Theorem}

\newtheorem{lemma}[nn]{Lemma}

\newtheorem{conjecture}[nn]{Conjecture}

\newtheorem{remark}[nn]{Remark}
\newtheorem*{acknowledgments*}{Acknowledgments}

%% file: introduction.tex
\section{Introduction}

Throughout, the positive integer $d \ge 1$ stands for the dimension of the ambient space $\R^d$. We consider problems for lattice polytopes in $\R^d$ and fix our underlying lattice to be $\Z^d$. We call the elements of $\Z^d$ \emph{lattice points} or \emph{lattice vectors}. A map $\phi : \R^d \to \R^d$ is called \emph{unimodular transformation} if $\phi$ is an affine bijection satisfying $\phi(\Z^d) = \Z^d$. By $o$ we denote the zero vector and $e_1,\ldots,e_d$ are the standard unit vectors in $\R^d$. 

A polytope $P \subseteq \R^d$ is called a \emph{lattice polytope} (with respect to the lattice $\Z^d$) if all vertices of $P$ belong to $\Z^d$. We will discuss the relationship of the Euclidean volume $\vol(P)$ and the number of interior lattice points for lattice polytopes $P$. Both these values are invariant up to unimodular transformations of $P$. 

Given positive integers $d, k \ge 1$, let $\cP^d(k)$ be the family of all $d$-dimensional lattice polytopes in $\R^d$ that have exactly $k$ interior lattice points and let $\cS^d(k)$ be the family of all simplices belonging to $\cP^d(k)$. 
It is known since the work of Hensley \cite{MR688412} that, for all positive integers $d, k$, the values 
\begin{align*}
	p(d,k) &:= \max \setcond{\vol(P)}{P \in \Pd{k}} & &\text{and} & 
	 s(d,k) &:= \max \setcond{\vol(S)}{S \in \Sd{k}}
\end{align*}
are finite. Exact determination of $p(d,k)$ and $s(d,k)$ is a long-standing problem \cite{MR688412, MR1138580, MR1996360, MR2967480, MR3318147} with additional motivation from integer optimization \cite[Section~2.7]{MR3318147} and toric geometry \cite[Section~2.6]{MR3318147}, cf. \cite{ben,Ambro}. 

In 1982, Zaks, Perles and Wills \cite{MR651251} introduced the simplex 
\begin{equation}
\label{eq:zpw:simplex}
S_{d,k} := \conv \bigl(o, s_1 e_1, \ldots, s_{d-1} e_{d-1}, (k+1) (s_d - 1)e_d \bigr),
\end{equation}
arising from the \emph{Sylvester sequence} $(s_i)_{i =1, 2, \ldots}$,
\begin{align}
& s_i := 
\begin{cases} 2, & \text{for} \ i=1,
\\ 1 + s_1 \ldots s_{i-1}  &  \text{for} \ i \ge 2.
\end{cases}
\label{eq:sylv:seq}
\end{align}
They observed that $S_{d,k}$ belongs to $\Sd{k}$ and has the volume 
\[
\vol(S_{d,k}) = (k+1) \frac{(s_d-1)^2}{d!}
\] 
of the asymptotic order $k 2^{2^{\Theta(d)}}$. This example shows that both $p(d,k)$ and $s(d,k)$ grow doubly-exponentially in the dimension $d$. By our main result, we want to partially verify the plausibility of the following conjecture that can be traced back to \cite{MR651251}. This explicit version can be found in \cite[Conj.~1.5]{arXiv:1612.08918}.

\begin{conjecture}
	\label{conj:bk} For every $d \ge 3$ and $k \ge 1$, one has 
	\[
	p(d,k) = s(d,k) = (k+1) \frac{(s_d-1)^2}{d!}.%\vol(S_{d,k}).
	\]
	Furthermore, unless $d=3$ and $k=1$, every polytope $P \in \Pd{k}$ satisying $\vol(P) = p(d,k)$ coincides with $S_{d,k}$ up to unimodular transformations.
\end{conjecture}

The conjecture claims that, in most cases, $S_{d,k}$ is an essentially unique volume maximizer in the family of all $d$-dimensional lattice polytopes with $k$ interior lattice points. It was recently verified for lattice simplices with $k=1$ \cite{MR3318147}. It is also known to hold in dimension $d=3$ for $k\le 2$ \cite{MR2760660,arXiv:1612.08918}. The previously best upper volume bounds were achieved by Pikhurko in 2001: \cite[Equation~(10)]{MR1996360} gives $s(d,k) \le k 2^{3d-2}15^{(d-1)2^{d+1}}/d!$. 
%Several other literature sources contain a hint to or a weaker version of Conjecture~\ref{conj:bk}, as discussed in \cite{arXiv:1612.08918}. Currently, it is not even known whether $s(d,k) = \vol(S_{d,k})$ is true, and so we want to make first steps towards verification of this equality. 
Improving upon his results, our main theorem confirms that the values $s(d,k)$ and $\vol(S_{d,k})$ are indeed very close to each other.

\begin{theorem}
	\label{thm:vol:simplices}
	Let $d, k$ be positive integers. Then
	\[
	s(d,k) \le  k (d+1) \frac{(s_d-1)^2}{d!}.
	\]
\end{theorem}

Theorem~\ref{thm:vol:simplices} implies $\vol(S_{d,k}) \le s(d,k) \le (d+1) \vol(S_{d,k})$. This determines $s(d,k)$ up to a linear factor in the dimension $d$. 

To prove Theorem~\ref{thm:vol:simplices}, we use the following result of Pikhurko. 
\begin{theorem}[{\cite[Lem.~5]{MR1996360}}]
	\label{thm:vol:simpl}
	Let $S$ be a $d$-dimensional simplex in $\R^d$, let $x$ be an interior lattice point of $S$ with barycentric coordinates $\beta_1 \ge \ldots \ge \beta_{d+1}$ and let $k= |\Z^d \cap \intr{S}|$. Then
	\begin{equation}
		\label{eq:vol:S:bary:better}
		\vol(S) \le \frac{k}{d! \, \beta_1 \cdots \beta_d}.
	\end{equation}
\end{theorem}

In view of Theorem~\ref{thm:vol:simpl}, for an arbitrary $S \in \Sd{k}$, we need to provide a lower bound on $\beta_1 \cdots \beta_d$ for some interior lattice point $x$ of $S$. Following Pikhurko, we choose $x$ to be the interior lattice point that maximizes the minimum barycentric coordinate of $x$.  In Section~\ref{gen:product:sum:sect} we derive so-called generalized product-sum inequalities for the barycentric coordinates of a point $x$ chosen in this way. They are the key tools that allow to bound $s(d,k)$ by considering purely analytical optimization problems. These inequalities are derived by adapting ideas from \cite{MR2967480} from the case $k=1$ (where they led to sharp volume bounds, see \cite{MR3318147}) to the case of an arbitrary $k \ge 1$.  In Section~\ref{opt:problem:s}, we formulate an optimization problem that involves generalized product-sum inequalities as constraints and which can be used for deriving lower bounds on $\beta_1 \cdots \beta_d$. In Sections~\ref{localizing} and \ref{univariate} we study this optimization problem and eventually prove Theorem~\ref{thm:vol:simplices}. 

Along the way, we improve two other results by Pikhurko \cite[Theorem~2 and Equation~(9)]{MR1996360}.

\begin{theorem}\label{pik-impro} Let $d,k$ be positive integers.
\begin{enumerate}[(a)]
	\item \label{bary:bound} For every $S \in \Sd{k}$ there exists an interior lattice point of $S$ all of whose all barycentric coordinates are at least $\frac{1}{(d+1) (s_{d+1}-1)}$.
	\item \label{p:bound} $p(d,k) \le \left(d (2d+1) (s_{2d+1}-1)\right)^d \, k$.
\end{enumerate}
\end{theorem}

The proof can be found at the end of Section~\ref{localizing}.
We expect that the order of the bound in (a) is sharp up to the reciprocal of a linear factor in the dimension.

Let us finish with two remarks. First, the authors wonder whether it is possible to use the present methods to replace in Theorem~\ref{thm:vol:simplices} the constant $d+1$ by a constant independent of the dimension $d$. The best such constant theoretically achievable would be $2$ (as one sees from $S_{d,1}$). Any further improvement would have to take the number $k$ of interior lattice points into account. However, it seems unclear how to introduce $k$ into the product-sum inequalities without significantly weakening the resulting bounds. Moreover, it was checked by Gabriele Balletti in the database of 3-simplices with $k=2$ \cite{arXiv:1612.08918} that (in contrast to the case $k=1$) it is impossible to use Theorem~\ref{thm:vol:simpl} to get a sharp volume bound -- hence, for sharpness a new approach would be needed. Secondly, the by far major challenge in this area of research is to replace $2d$ in Theorem~\ref{pik-impro}(b) by a linear function in $d$. This doubling of the dimension results from a clever trick by Pikhurko \cite[proof of Theorem~4]{MR1996360} to reduce the problem from one for polytopes to one for simplices. However, it is a wide open question how such a reduction could be done without increasing the dimension.

%% file: background.tex
\paragraph{Basic notation and terminology.}

By $\N = \{1,2,3,\ldots\}$ we denote the set of (positive) natural numbers. For a non-negative integer $t$, let $[t]:=\{1,\ldots,t\}$ if $t > 0$ and $[t]:=\emptyset$ if $t=0$.  The cardinality of a set $X$ is denoted by $|X|$.  For $X \subseteq \R^d$, the dimension $\dim(X)$ of $X$ is defined as the dimension of the affine hull of $X$. The volume (that is, the Lebesgue measure) of a Lebesgue measurable set $X$ is denoted by $\vol(X)$. We use the standard scaling of $\vol$ with $\vol([0,1]^d) = 1$. 

The interior and the convex hull of $X \subseteq \R^d$ are denoted by $\intr{X}$ and $\conv(X)$, respectively. The standard basis vectors of $\R^d$ are denoted by $e_1,\ldots,e_d$. For $a, b \in \R^d$ we use $[a,b]$ to denote the convex hull of $\{a,b\}$. If $a \ne b,$ the set $[a,b]$ is the line segment joining $a$ and $b$. A polytope in $\R^d$ is the convex hull of a finite subset of $\R^d$. For more information on polytopes see \cite{MR1311028}. The set of all vertices of a polytope $P$ will be denoted by $\vertset{P}$. A polytope $S$ is called a simplex if the vertices of $S$ are affinely independent. 

If $S$ is a $d$-dimensional simplex in $\R^d$ with vertices $v_1,\ldots,v_{d+1}$, then every $x \in \R^d$ has a unique representation as the affine combination of $v_1,\ldots,v_{d+1}$, that is, there exist $\beta_1,\ldots,\beta_{d+1} \in \R$ uniquely determined by $S$ and $x$ such that 
\begin{align*}
x & = \beta_1 v_1 + \cdots + \beta_{d+1} v_{d+1},
\\ 1 & = \beta_1 + \cdots + \beta_{d+1}.
\end{align*}
The values $\beta_1,\ldots,\beta_{d+1}$ are called the \emph{barycentric coordinates} of $x$ with respect to $S$. Clearly, $x \in S$ if and only if all the barycentric coordinates are non-negative and $x \in \intr{S}$ if and only if all the barycentric coordinates are strictly positive.

%% file: prod-sum.tex
\section{Generalized product-sum inequalities}

\label{gen:product:sum:sect}

In \cite{MR2967480} it was shown that, for every $S \in \Sd{1}$, the barycentric coordinates $\beta_1 \ge \ldots \ge \beta_{d+1}$ of the unique interior lattice point of $S$ satisfy the inequalities
\begin{align}
\label{eq:product:sum}
 \prod_{i = 1}^t \beta_i & \le \sum_{j =t+1}^{d+1} \beta_j & & \forall t \in [d].
\end{align}
We call \eqref{eq:product:sum} the \emph{product-sum inequalities}. In \cite{MR3318147}, $s(d,1)$ %and the maximum lattice asymmetry of polytopes in $\Sd{1}$ 
was determined using the product-sum inequality and Theorem~\ref{thm:vol:simpl}. 

For $\Sd{k}$ with an arbitrary $k \ge 1$ the situation is more difficult, because there is a freedom in the choice of $x \in \Z^d \cap \intr{S}$ and it is crucial to choose $x$ appropriately. Pikhurko \cite{MR1996360} suggested to choose a point $x \in \Z^d \cap \intr{S}$ with barycentric coordinates $\beta_1, \ldots. \beta_{d+1}$ maximizing $\gamma:=\min \{\beta_1,\ldots,\beta_{d+1}\}$. He used a lemma of Lagarias and Ziegler \cite[Lemma~2.1]{MR1138580} to get lower bounds on $\gamma$ in terms of $d$.

Our approach to generating inequalities for barycentric coordinates combines ideas from \cite{MR2967480} and \cite{MR1996360}. We use an interior lattice point maximizing the minimum barycentric coordinate as in \cite{MR1996360} and then, for this point, we derive generalized product-sum inequalities by adapting the arguments from \cite{MR2967480}. For a vector $y=(y_1,\ldots,y_n) \in \R^n$ ($n \in \N$), we introduce its maximum norm $\|y\|_\infty:= \max \{ |y_1|,\ldots,|y_n| \}$. 

\begin{lemma}[Determinant lemma; see \cite{MR2967480}]
	\label{determinant:lemma}
	Let $n \in \N$ and $A \in \R^{n \times n}$. If $0<|\det(A)|<1$, then there exists $y \in \Z^n \setminus \{o\}$ satisfying $\|A y \|_\infty<1$.
\end{lemma}

\begin{theorem}[Generalized product-sum inequalities]
	\label{thm:prod:sum}
	Let $S \subseteq \R^d$ be $d$-dimensional lattice simplex in $\R^d$ and let $x$ be an interior lattice point of $S$ with barycentric coordinates $\beta_1 \ge \ldots \ge \beta_{d+1}$ having the property that all barycentric coordinates of every other interior lattice point of $S$ are not smaller than $\beta_{d+1}$. Then, for every  $t \in [d]$,
	one has
	\begin{align}
	\label{eq:gen:prod:sum}
	\prod_{i =1}^t (\beta_i - \beta_{d+1}) \le \sum_{j=t+1}^{d+1} \beta_j.
	\end{align}
\end{theorem}
\begin{proof}
	If $\beta_i = \beta_{d+1}$ for some $i \in [t]$, the assertion is trivially fulfilled. So, we assume that $\beta_i > \beta_{d+1}$ holds for every $i \in [t]$. Let $v_1,\ldots,v_{d+1}$ be the vertices of $S$ such that $x = \beta_1 v_1 + \cdots + \beta_{d+1} v_{d+1}$. Throughout the proof, we consider integer variables $m,m_1,\ldots,m_t \in \Z$. If 
	\begin{align}
	\label{eq:proper:ms} 
	m & >0 & &\text{and} & m & =m_1+ \cdots + m_t, 
	\end{align}
	then 
	\[
		r = \sum_{i=1}^t \frac{m_i}{m} v_i
	\]
	is a point in the affine hull of the points $v_1,\ldots,v_t$ and 
	\[
	q = (m+1) x - m r
	\]
	is a lattice point on the line passing through $r$ and $x$. Our proof approach is to assume that \eqref{eq:gen:prod:sum} is not fulfilled. Under this assumption we derive a contradiction to the choice of $x$ by showing that for some $m,m_1,\ldots,m_t$ satisfying \eqref{eq:proper:ms}, the minimum barycentric coordinate of the point $q$ is strictly larger than $\beta_{d+1}$ (necessarily, $q$ is then an interior lattice point of $S$). 
	
	For this, let us note that whenever \eqref{eq:proper:ms} and 
	\begin{align}
	\label{eq:bar:coord:q}
	(m+1) \beta_i - m_i & > \beta_{d+1} &  \forall & i \in [t]
	\end{align}
	hold, the minimum barycentric coordinate of $q$ is strictly larger than $\beta_{d+1}$. Indeed, $(m+1) \beta_i -m_i$ with $i \in [t]$ are the barycentric coordinates of $q$ with respect to the vertices $v_1,\ldots,v_t$. They are strictly larger than $\beta_{d+1}$. Since $m\ge 1$, one has $(m+1) \beta_j \ge 2 \beta_{d+1} > \beta_{d+1}$, and so the remaining barycentric coordinates $(m+1) \beta_j$ with $j > t$ are also strictly larger than $\beta_{d+1}$. Summing up, if \eqref{eq:proper:ms} and \eqref{eq:bar:coord:q} hold, we get a contradiction. 
	
	We reformulate \eqref{eq:bar:coord:q} as $\frac{1}{\beta_i-\beta_{d+1}} ( m_i - \beta_i m) < 1$.  Since we want to use Lemma~\ref{determinant:lemma}, we will consider the stronger condition $|\frac{1}{\beta_i - \beta_{d+1}} (m_i - \beta_i m ) |< 1$ on the variables $m_1,\ldots,m_t,m$. Taking into account that   $m,m_1,\ldots,m_t$ are integer variables, the condition $m = m_1 +\cdots +m_t$ can be reformulated as the strict inequality $|-m_1 - m_2 - \cdots -m_t + m| < 1$. Altogether, we have introduced a system of $t+1$ strict inequalities for the $t+1$ integer variables  $m_1,\ldots,m_t,m$, which can be formulated as
	\begin{align}
	\label{eq:max:norm}
	\left\| 
	\underbrace{\begin{pmatrix} 
		\frac{1}{\beta_1-\beta_{d+1}} &  &  & -\frac{\beta_1}{\beta_1 - \beta_{d+1}}
		\\ & \ddots &  & \vdots
		\\ & & \frac{1}{\beta_t-\beta_{d+1}} & -\frac{\beta_t}{\beta_t - \beta_{d+1}}
		\\ -1 & \cdots & -1 & 1
		\end{pmatrix}}_{=:A}
	\begin{pmatrix}
	m_1 \\ \vdots \\ m_t \\ m
	\end{pmatrix}
	\right\|_\infty < 1.
	\end{align}
	Performing elementary row operations, we compute $\det(A)$: 
	\begin{align*}
	\det(A) & %
	= \frac{1}{\prod_{i=1}^t (\beta_i - \beta_{d+1})}  \begin{vmatrix} 
	\phantom{-}1 &  &  & -\beta_1
	\\ & \ddots &  & \vdots
	\\ & & \phantom{-}1 & - \beta_t
	\\ -1 & \cdots & -1 & 1
	\end{vmatrix}
	\\ & = 
	\frac{1}{\prod_{i=1}^t (\beta_i - \beta_{d+1})}  \begin{vmatrix} 
	\phantom{-}1 &  &  & -\beta_1
	\\ & \ddots &  & \vdots
	\\ & & \phantom{-}1 & - \beta_t
	\\ \phantom{-}0 & \cdots & \phantom{-}0 & 1- \sum_{i=1}^t \beta_i
	\end{vmatrix}
	\\ & = \frac{1 - \sum_{i=1}^t \beta_i}{ \prod_{i=1}^t (\beta_i -\beta_{d+1})} 
	= \frac{\sum_{j=t+1}^{d+1} \beta_j}{ \prod_{i=1}^t (\beta_i -\beta_{d+1})}.
	\end{align*}
	
	From this we see that $\det(A) > 0$ and, since we assumed that \eqref{eq:gen:prod:sum} is not fulfilled,  we also have $\det(A) < 1$. Thus, by Lemma~\ref{determinant:lemma} applied to \eqref{eq:max:norm}, there exists $(m_1,\ldots,m_t,m) \in \Z^{t+1} \setminus \{o\}$ satisfying \eqref{eq:max:norm}. The value $m$ cannot be zero, since otherwise \eqref{eq:max:norm} would imply that $m_1,\ldots,m_t$ are all zero, too. Since $m \ne 0$, possibly replacing $(m_1,\ldots,m_t,m)$ by $-(m_1,\ldots,m_t,m)$, we can assume $m>0$. We have found $m_1,\ldots,m_t,m$ satisfying \eqref{eq:proper:ms} and \eqref{eq:bar:coord:q}. This contradicts the choice of $x$, as described above. 
\end{proof}

\begin{remark}{\rm 
	Another generalization of the product-sum inequalities was obtained in the master thesis of Brunink \cite[Satz~4.1]{Brunink:MasterThesis:2016}. Brunink's inequalities are valid with respect to \emph{every} $x \in \Z^d \cap \intr{S}$ and so they are necessarily weaker than our inequalities, which are valid for $x$ that maximizes the minimum barycentric coordinate.}
\end{remark}

%% file: opt-prod-sum.tex
\section{An optimization problem for bounding $s(d,k)$}

\label{opt:problem:s}

We weaken \eqref{eq:gen:prod:sum} to the more convenient inequality
\begin{equation}
	\label{eq:PS}	
 \prod_{i=1}^t (\beta_i - \beta_{d+1}) \le \sum_{j=t+1}^{d+1} (\beta_j - \beta_{d+1}) + (d+1) \beta_{d+1}
\end{equation} and use this inequality to introduce the following optimization problem on a compact subset of $\R^{d+1}$:
\begin{align}
	\tau_d := \min \biggl\{ \beta_1 \cdots \beta_d \,:\; & \biggr. 
(\beta_1, \ldots, \beta_{d+1}) \in \R^{d+1}, \label{opt:problem}
	\\ & \beta_1 \ge \ldots \ge \beta_{d+1} \ge 0, \; \beta_1 + \cdots + \beta_d = 1, 
		\nonumber
	\\ & \biggl. \prod_{i=1}^t (\beta_i - \beta_{d+1}) \le \sum_{j=t+1}^{d+1} (\beta_j - \beta_{d+1}) + (d+1) \beta_{d+1} & &  \forall t \in [d] \biggr\}. \nonumber
\end{align}

During the analysis of this problem, we use the standard optimization terminology, such as \emph{feasible solution}, \emph{optimal solution} and \emph{optimal value}. Inequality $\beta_t \ge \beta_{t+1}$ for a given $t \in [d]$ will be denoted by $\ORD(t)$, where $\ORD$ stands for `ordering', while inequality \eqref{eq:PS} for a given $t \in [d]$ will be denoted by $\PS(t)$, where $\PS$ stands for `product-sum'. We will use $\ORD(0)$ to denote the inequality $1 \ge \beta_1$.

We start our analysis of \eqref{opt:problem} by observing some basic properties of feasible solutions. 

\begin{lemma}
	\label{lem:one:is:strict}
	Let $\beta=(\beta_1,\ldots,\beta_{d+1})$ be a feasible solution of the optimization problem \eqref{opt:problem}. Then the following hold:
	\begin{enumerate}[(a)]
		\item \label{all:positive} For every $t \in [d+1]$, $\beta_t > 0$.
		\item \label{ord:ps:strict} For every $t \in [d]$, $\ORD(t)$ or $\PS(t)$ is strict.
		\item \label{rel:sylv} If for some $\ell \in [d]$, $\PS(i)$ holds with equality for every $i \in [\ell]$, then $\beta_i = \frac{1}{s_i} + \beta_{d+1}$ for every $i \in [\ell]$. 
	\end{enumerate}
\end{lemma}
\begin{proof}
	\eqref{all:positive}: If (a) was false, we would have $\beta_1 \ge  \cdots \ge \beta_\ell > \beta_{\ell+1} = \cdots = \beta_{d+1} =0$ for some $\ell \in [d]$. With this choice of $\ell$, the right-hand side of $\PS(\ell)$ would be zero and the left-hand side would be strictly positive, which is a contradiction.
	
	\eqref{ord:ps:strict}: If $\ORD(t)$ holds with equality, then $\beta_t=\beta_{t+1}$ and so $\beta_t-\beta_{d+1} = \beta_{t+1} - \beta_{d+1}$. Then, in $\PS(t)$ the left hand side is at most $\beta_t - \beta_{d+1}$, while the right hand side is strictly larger than $\beta_{t+1} - \beta_{d+1} = \beta_t - \beta_{d+1}$ (as $\beta_{d+1} > 0$ by (a)). Thus, \eqref{eq:PS} is strict.
	
	\eqref{rel:sylv}: The equality $\beta_i = \frac{1}{s_i} + \beta_{d+1}$ can be reformulated as $\frac{1}{\beta_i - \beta_{d+1}} = s_i$. We can prove \eqref{rel:sylv} by induction on $i \in [\ell]$. For $i=1$, the equality in $\PS(1)$ gives $\beta_1 - \beta_{d+1}=(1-\beta_1)+\beta_{d+1}$, so $\frac{1}{\beta_1- \beta_{d+1}} = 2 = s_1$. Assume that $\ell \ge i > 1$ and $\frac{1}{\beta_j - \beta_{d+1}} = s_j$ for every $j < i$. Since $\PS(i)$ holds with equality we have 
	\[\prod_{j=1}^i (\beta_j - \beta_{d+1})  + \sum_{j=1}^i (\beta_j - \beta_{d+1}) = 1.\]
	Factoring out $\beta_i-\beta_{d+1}$ and using the induction hypothesis, we arrive at
	\[
	(\beta_i-\beta_{d+1}) \left( \frac{1}{s_1 \cdots s_{i-1}} + 1 \right) + \frac{1}{s_1} + \cdots + \frac{1}{s_{i-1}}  = 1.
	\]
	Using the well-known equality $\frac{1}{s_1} + \cdots + \frac{1}{s_{i-1}} + \frac{1}{s_1 \cdots s_{i-1}} = 1$, we get $\frac{1}{\beta_i - \beta_{d+1}} = s_1 \cdots s_{i-1} + 1 = s_i$, as desired.
\end{proof}

\begin{theorem}
	\label{thm:lower:bound:and:s}
	Let $d, k \in \N$. Then $\tau_d>0$ and
	\[
		s(d,k) \le \frac{k}{d! \, \tau_d}
	\]
\end{theorem}
\begin{proof}
	By Lemma~\ref{lem:one:is:strict}\eqref{all:positive}, $\tau_d > 0$. The asserted bound on $s(d,k)$ follows by combining Theorems~\ref{thm:vol:simpl} and \ref{thm:prod:sum}. 
\end{proof}

\section{Localizing optimal solutions}

\label{localizing}

Our aim now is to provide a possibly exact description of optimal solutions of \eqref{opt:problem}, in order to determine $\tau_d$ as precisely as possible. 

\begin{lemma}
	\label{lem:vol:opt:problem}
	Let $\beta=(\beta_1,\ldots,\beta_{d+1})$ be an optimal solution of \eqref{opt:problem}. Then there exists $\ell \in \{0,\ldots,d\}$ such that:
	\begin{enumerate}[(a)]
		\item $\ORD(t)$ holds with equality for $\ell + 1 \le t \le d$,
		\item $\PS(t)$ holds with equality for $1 \le t \le \ell -1$.
	\end{enumerate}
\end{lemma}
\begin{proof}
 We set $\beta_0 := 1$. Choose $\ell \in \{0,\ldots,d\}$ such that 
	\[
		\beta_\ell > \beta_{\ell+1} = \cdots = \beta_{d+1}.
	\]
With this choice of $\ell$, (a) clearly holds and it remains to show (b). 
	
We first show that $\ORD(t)$ is also strict for each $t \le \ell-1$. If the latter is not the case, then we can find indices $a$ and $b$ with $1 \le a < b \le \ell$ such that 
	\[
		\beta_{a-1} > \beta_a = \cdots = \beta_b > \beta_{b+1}
	\]
	We can perturb $\beta$ by changing $\beta_a$ to $\beta_a+\eps$ and $\beta_b$ to $\beta_b - \eps$ for a small $\eps>0$. This change makes the value of the objective function smaller because the product $\beta_a \beta_b$ in the objective function gets smaller in view of the inequality
	\begin{equation}
		\label{eq:product:gets:smaller}
		(\beta_a + \eps)(\beta_b - \eps) < \beta_a \beta_b.
	\end{equation}
	On the other hand, our perturbation keeps $\beta$ a feasible solution. All $\ORD(1),\ldots,\ORD(d)$ remain fulfilled.  Inequalities $\PS(t)$ with $t < a$ remain fulfilled because, both sides of these inequalities remain unchanged. Inequalities $\PS(t)$ with $a \le t < b$ remain valid if $\eps>0$ is small, because in view of Lemma~\ref{lem:one:is:strict}(b) they were all strict. Inequalities $\PS(t)$ with $t \ge b$ remain valid, because by \eqref{eq:product:gets:smaller} the left hand side of these inequalities gets smaller, while the right hand side remains unchanged. This is a contradiction to the choice of $\beta$, because we have verified that a small perturbation of $\beta$ yields a feasible solution with a strictly smaller value of the objective function. This proves
	\[
		\beta_1 > \cdots > \beta_{\ell+1} = \cdots = \beta_{d+1}. 
	\]
	To verify (b), we choose an arbitrary $t$ with $1 \le t \le \ell -1$ and show that $\PS(t)$ holds with equality. If $\PS(t)$ is strict, we can perturb $\beta$ by changing $\beta_t$ to $\beta_t + \eps$ and $\beta_{t+1}$ to $\beta_{t+1} - \eps$ for a small $\eps>0$. This perturbation makes the value of the objective function strictly smaller. Indeed, the objective function contains the product $\beta_t \beta_{t+1}$, which gets smaller in view of 
	\begin{equation}
		\label{eq:another:product:gets:smaller}
		(\beta_t + \eps) (\beta_{t+1} - \eps) < \beta_t \beta_{t+1}.
	\end{equation}
	
	Clearly, $\ORD(1),\ldots,\ORD(d)$ remain valid. $\PS(1),\ldots,\PS(t-1)$ remain valid because both sides of the inequality are unchanged. $\PS(t)$ remains valid because it was strict. $\PS(t+1),\ldots,\PS(d)$ remain valid because the left-hand side gets smaller in view of \eqref{eq:another:product:gets:smaller} and the right-hand side does not change. This implies that a small perturbation of $\beta$ is a feasible solution with a strictly smaller value of the objective function. This contradicts the choice of $\beta$ and shows that $\PS(t)$ holds with equality. 
\end{proof}

For $\ell \in [d]$ let us define the following auxiliary function:
		\begin{align*}
			f_\ell(\alpha)   :=& 
			 \left( \prod_{i=1}^{\ell-1} \left(\frac{1}{s_i} + \alpha \right) \right) \left(\frac{1}{s_{\ell}-1} - d \alpha \right) \alpha^{d-\ell} \quad \text{ for } \alpha \in \R.
		\end{align*}

\begin{lemma}
	\label{lem:reduction:univariate}
	Let $\beta = (\beta_1,\ldots,\beta_{d+1})$ be an optimal solution of \eqref{opt:problem}. Then there exists some $\ell \in [d]$ such that the following conditions hold: 
	\begin{enumerate}[(a)]
		\item \label{expr:objective} The value $\beta_1 \cdots \beta_d$ of the objective function can be expressed using $\beta_{d+1}$ as 
		\[
			\beta_1 \cdots \beta_d = f_\ell(\beta_{d+1}).
		\]
		\item \label{expr:range:smallest:bar} The value $\beta_{d+1}$ satisfies 
		\[
			\frac{1}{(d+1)(s_{\ell+1} -1)} \le \beta_{d+1} \le \frac{1}{(d+1)(s_\ell-1)}.
		\]
	\end{enumerate}
\end{lemma}
\begin{proof}
	If $\beta_1 =  \cdots = \beta_{d+1}$, then $\beta_1 = \cdots = \beta_{d+1} = \frac{1}{d+1}$ and the assertions hold with $\ell=1$ (note that $f_1(\alpha) = (1 - d \alpha) \alpha^{d-1}$). Otherwise, we choose $\ell \in \{0,\ldots,d\}$ as in Lemma~\ref{lem:vol:opt:problem}. Since $\beta_1,\ldots,\beta_{d+1}$ are not all equal, we have $\ell > 0$. Using Lemma~\ref{lem:vol:opt:problem} and Lemma~\ref{lem:one:is:strict}(c), we arrive at the equalities
	\begin{align*}
		\beta_i & = \frac{1}{s_i} + \beta_{d+1} & & \text{for}  \ i \le \ell-1,
		\\ \beta_i & = \beta_{d+1} & & \text{for} \ i \ge \ell + 1.
	\end{align*}
	This expresses all $\beta_i$ with $i \in [d+1] \setminus \{\ell\}$ in terms of $\beta_{d+1}$. From $\beta_1 + \cdots + \beta_{d+1} = 1$, we also obtain a representation of $\beta_\ell$ using $\beta_{d+1}$: 
	\[
		\beta_\ell = 1 - \sum_{i=1}^{\ell -1} \left( \frac{1}{s_i} + \beta_{d+1} \right) - (d + 1 - \ell ) \beta_{d+1} = 1 - \sum_{i=1}^{\ell-1} \frac{1}{s_i} - d \beta_{d+1} = \frac{1}{s_\ell - 1} - d \beta_{d+1}.
	\]
	This yields \eqref{expr:objective}.	For verifying \eqref{expr:range:smallest:bar}, it suffices to observe that in view of the above representations of $\beta_1,\ldots,\beta_d$ in terms of $\beta_{d+1}$, inequality $\ORD(\ell)$ and $\PS(\ell)$ amount after some straightforward computation to $\beta_{d+1} \le \frac{1}{(d+1) (s_\ell-1)}$ and $\beta_{d+1} \ge \frac{1}{(d+1)(s_{\ell+1}-1)}$, respectively.
\end{proof}

\begin{proof}[Proof of Theorem~\ref{pik-impro}]
Assertion~\eqref{bary:bound} follows from Theorem~\ref{thm:prod:sum} and Lemma~\ref{lem:reduction:univariate}(b). 

We prove \eqref{p:bound}. Let $\gamma> 0$ be such that every lattice simplex of dimension at most $2d$ having interior lattice points contains an interior lattice point all of whose barycentric coordinates are at least $\gamma$. In the proof of Theorem~4 of \cite{MR1996360}, Pikhurko shows that  every $d$-dimensional lattice polytope $P$ having interior lattice points contains an interior lattice point $w$ with $P - w \subseteq \sigma (w-P)$, where $\sigma := \frac{d}{\gamma} - 1$. By \eqref{bary:bound}, we can fix $\gamma := \frac{1}{(2 d + 1) (s_{2d+1} - 1)}$. Lagarias and Ziegler \cite[Thm.~2.5]{MR1138580} observe that, if $P \in \Pd{k}$, then $\vol(P) \le (1+\sigma)^d k$ (see also \cite[Eq.~(8)]{MR1996360}). This yields the desired estimate.  
\end{proof}

\section{Univariate optimization problems and the conclusion}

\label{univariate}

By Lemma~\ref{lem:reduction:univariate}, 
\[
	\tau_d \ge \min \setcond{f_\ell^\ast}{\ell \in [d]},
\] 
where
\[
	f^\ast_\ell:=\min \setcond{f_\ell(\alpha)}{\frac{1}{(d+1)(s_{\ell+1} -1)} \le \alpha \le \frac{1}{(d+1)(s_\ell-1)}}.
\]
Thus, we have relaxed optimization problem \eqref{opt:problem} with $d+1$ variables to $d$ univariate optimization problems. We will determine a common lower bound on the optimal values $f^\ast_1,\ldots,f^\ast_d$.

\begin{lemma}
	\label{lem:bound:univariate}
	Let $d \ge 4$. Then 
	\[
		\tau_d \ge \frac{1}{(d+1) (s_d-1)^2}.
	\]

\end{lemma}
\begin{proof}
	By Lemma~\ref{lem:reduction:univariate}, it suffices to verify
	\[
	f^\ast_\ell \ge \frac{1}{(d+1) (s_d-1)^2}
	\]
	for every $\ell \in [d]$. 
	
	Each $f_\ell(\alpha)$ with $\ell \in [d]$ is a polynomial of degree $d$ in $\alpha$, whose all roots are real and
	exactly one root is positive. Applying Rolle's theorem it is
	straightforward to see that all roots of the derivative of
	$f_\ell(\alpha)$ are real too and that the derivative has at most one
	positive root. Taking into account the asymptotics of $f_\ell(\alpha)$ at
	infinity, the latter observations show that $f_\ell(\alpha)$ is unimodal on the segment
	$\left[\frac{1}{(d+1)(s_{d+1}-1)},\frac{1}{(d+1)(s_d-1)}\right]$ and so it attains
	its minimum over this segment at one of its endpoints. This allows us
	to compute $f^\ast_\ell$ for concrete values of $d$. It is thus
	straightforward to check that our assertion is true for $d=4$. Assume $d \ge 5$. 

	For $\frac{1}{(d+1)(s_{\ell+1} -1)} \le \alpha \le \frac{1}{(d+1)(s_\ell-1)}$, one has 
	\begin{align*}
	f_\ell(\alpha) & 
	\ge \frac{1}{s_1 \cdots s_{\ell-1}} \left( \frac{1}{s_\ell-1} -d \alpha \right) \alpha^{d-\ell} 
	\\ & = \frac{1}{s_{\ell}-1} \left( \frac{1}{s_\ell-1} -d \alpha \right) \alpha^{d-\ell}
	\\ & \ge \frac{1}{s_{\ell}-1} \left( \frac{1}{s_\ell-1} - \frac{d}{(d+1) (s_\ell-1)} \right) \left( \frac{1}{(d+1)(s_{\ell+1} -1)}\right)^{d-\ell}
	\\ & =  \frac{1}{(s_{\ell}-1)(d+1)(s_\ell - 1)}  \left( \frac{1}{(d+1)(s_{\ell+1} -1)}\right)^{d-\ell}.
%	\frac{1}{(s_\ell-1)^2 (d+1)^{d-\ell+1} (s_{\ell+1} - 1)^{d-\ell}}.
	\end{align*}
	We have thus derived the inequality $f_\ell^\ast \ge \frac{1}{y_\ell}$, 	where
	\begin{equation*}
		y_{\ell}:=(s_\ell-1)^2 (d+1)^{d-\ell+1} (s_{\ell+1} - 1)^{d-\ell}.
	\end{equation*}
Note that
\[\frac{1}{y_d} = \frac{1}{(d+1) (s_d-1)^2}\]
	Hence, we have to show $f_\ell^\ast \ge \frac{1}{y_d}$ for $\ell \in [d-1]$. 
	
	We first verify  $y_1 \le \cdots \le y_{d-2}$.	That is, for $\ell \in [d-3]$, we need to check $y_\ell \le y_{\ell+1}$. One can check using the relation $s_{i+1} - 1 = s_i (s_i-1)$ valid for all $i \ge 1$ that the latter inequality is equivalent to 
	\begin{align*}
	%		& & y_\ell & \le y_{\ell+1}
	%		\\ & \Leftrightarrow &  (s_l-1)^2 (d+1)^{d-l+1} (s_{l+1}-1)^{d-l} &  \le (s_{l+1} - 1)^2 (d+1)^{d-l} (s_{l+2}-1)^{d-l-1}
	%		\\ & \Leftrightarrow &  (s_l-1)^2 (d+1) (s_{l+1}-1)^{d-l} & \le (s_{l+1} - 1)^2 (s_{l+2}-1)^{d-l-1}
	%		\\ & \Leftrightarrow &  (s_l-1)^2 (d+1) (s_{l+1}-1)^{d-l-2} & \le (s_{l+2}-1)^{d-l-1}
	%		\\ & \Leftrightarrow &  (s_l-1)^2 (d+1) (s_{l+1}-1)^{d-l-2} & \le s_{l+1}^{d-l-1} (s_{l+1}-1)^{d-l-1}
	%		\\ & \Leftrightarrow &  (s_l-1)^2 (d+1)  & \le s_{l+1}^{d-l-1} (s_{l+1}-1)
	%		\\ & \Leftrightarrow &  (s_l-1)^2 (d+1) & \le s_{l+1}^{d-l-1} s_l (s_l-1)
			%\\ & \Leftrightarrow &   
		(s_\ell-1) (d+1) & \le s_{\ell+1}^{d-\ell-1} s_\ell 
	\end{align*}
	We will verify the slightly stronger inequality
	\begin{equation}
		 d+ 1 \le s_{\ell+1}^{d- \ell -1} \label{eq:d+1:sylv}
	\end{equation}
	for each $\ell \in [d-3]$. For this, let us first observe that the right hand side of \eqref{eq:d+1:sylv} is increasing in $\ell$ in the range $1 \le \ell \le d-3$. Indeed, if $\ell \in [d-4]$, one has $s_{\ell+1}^{d- \ell -1} \le s_{\ell+2}^{d-\ell -2}$, which can be verified using the estimate $s_{\ell+2} \ge s_{\ell+1} (s_{\ell+1} - 1)$. This implies that it suffices to check  \eqref{eq:d+1:sylv} for $\ell =1$, in which case it amounts to \(d+1 \le 3^{d-2}\), which can be easily verified by induction on $d \ge 4$.  
	
	We have shown $y_1 \le \ldots \le y_{d-2}$. Let us now check $f_\ell^\ast \ge \frac{1}{y_d}$ for all $\ell \in [d-2]$ by showing $y_{d-2} \le y_d$. Inequality $y_{d-2} \le y_d$ amounts to $(s_{d-2}-1)(d+1) \le s_{d-1}$. The stronger inequality $d+1 \le s_{d-2}$ can be checked by induction on $d \ge 5$.
	
	To conclude the proof, it remains to verify $f_{d-1}^\ast \ge \frac{1}{y_d}$. For this, we estimate $f_{d-1}(\alpha)$ again from below by a quadratic function in $\alpha$: 
	\[
		f_{d-1}(\alpha) \ge \frac{1}{s_{d-1} -1} \left( \frac{1}{s_{d-1}-1} -d \alpha \right) \alpha.
	\]
	 The quadratic function on the right hand side is a concave downward parabola. Hence, it takes its minimum on the segment $\left[ \frac{1}{(d+1)(s_d-1)}, \frac{1}{(d+1)(s_{d-1} -1)} \right]$ at one of the two endpoints. So, we are left with checking $f_{d-1}\left(\frac{1}{(d+1)(s_d-1)}\right) \ge \frac{1}{(d+1)(s_d-1)^2}$ and $f_{d-1}\left(\frac{1}{(d+1)(s_{d-1}-1)}\right) \ge \frac{1}{(d+1)(s_d-1)^2}$. This can be easily verified.
\end{proof}

\begin{proof}[Proof of Theorem~\ref{thm:vol:simplices}]
	For $d=1$, the inequality is trivial. For $d=2$, the inequality is true in view of Scott's result \cite{MR0430960} which implies $s(2,k) = 2 (k+1)$ for $k \ge 2$ and $s(2,1) = 4{.}5$. For $d=3$, the asserted inequality is weaker than Pikhurko's inequality $s(d,k) \le 14{.}106 k$; see \cite{MR1996360}. In the case $d \ge 4$, the asserted inequality follows from  Theorem~\ref{thm:lower:bound:and:s} and Lemma \ref{lem:bound:univariate}.
\end{proof}